\documentclass{article}

\usepackage{amsmath}
\usepackage{amsfonts}
\usepackage{amsthm}
\usepackage{amssymb}
\usepackage{mathrsfs}
\usepackage{mathtools}
\usepackage{geometry}
\usepackage{graphicx}
\usepackage{hyperref}
\hypersetup{
    colorlinks=true,       
    linkcolor=red,          
    citecolor=blue,       
}

\renewcommand{\Re}{\text{\rm Re}\,}
\renewcommand{\Im}{\text{\rm Im}\,}
\newcommand{\PT}{\mathcal{PT}}

\newcommand{\T}{\mathcal{T}}
\newcommand{\ii}{{\rm i}}
\newcommand{\dd}{{\rm d}}
\newcommand{\ee}{{\rm e}}
\newcommand{\hilb}{\mathscr{H}}
\renewcommand{\arg}{\mathrm{arg}}
\newcommand{\Dom}{\mathrm{Dom}}

\newcommand{\beq}{\begin{equation}}
\newcommand{\eeq}{\end{equation}}
\newcommand{\beqs}{\begin{equation*}}
\newcommand{\eeqs}{\end{equation*}}
\newcommand{\bal}{\begin{aligned}}
\newcommand{\eal}{\end{aligned}}
\newcommand{\RR}{\mathbb{R}}
\newcommand{\CC}{\mathbb{C}}

\newcommand{\eps}{\varepsilon}

\theoremstyle{plain}
\newtheorem{thm}{Theorem}[section] 

\newtheorem{lem}[thm]{Lemma}

\theoremstyle{definition}

\theoremstyle{remark}

\usepackage[affil-it,auth-sc-lg]{authblk}

\title{On the pseudospectrum of the harmonic oscillator with imaginary cubic potential}
\date{\today}
\author{Radek Nov\'{a}k\thanks{Electronic address: \texttt{novakra9@fjfi.cvut.cz}}}
\affil{Department of Physics, Faculty of Nuclear Sciences and Physical Engineering, Czech Technical University in Prague, B\v rehov\' a 7, 115 19, Prague, Czech Republic}
\affil{Department of Theoretical Physics, Nuclear Physics Institute, Academy of Sciences of the Czech Republic, Hlavn\' i 130, 250 68 \v Re\v z near Prague, Czech Republic}
\affil{Laboratoire de Math\'{e}matiques Jean Leray, 9, Universit\'{e} de Nantes, 2 rue de la Houssini\`{e}re, 443 22 Nantes Cedex 3, France}
\numberwithin{equation}{section}
\providecommand{\keywords}[1]{\textbf{\textit{Key words: }} #1\\[0.5em]}
\providecommand{\classification}[1]{\textbf{\textit{Mathematics Subject Classification (2010): }} #1}
\begin{document}
%
\maketitle
\begin{abstract}
We study the Schr\"{o}dinger operator with a potential given by the sum of the potentials for harmonic oscillator and imaginary cubic oscillator and we focus on its pseudospectral properties. A summary of known results about the operator and its spectrum is provided and the importance of examining its pseudospectrum as well is emphasized. This is achieved by employing scaling techniques and treating the operator using semiclassical methods. The existence of pseudoeigenvalues very far from the spectrum is proven, and as a consequence, the spectrum of the operator is unstable with respect to small perturbations and the operator cannot be similar to a self-adjoint operator via a bounded and boundedly invertible transformation. It is shown that its eigenfunctions form a complete set in the Hilbert space of square-integrable functions; however, they do not form a Schauder basis.
\end{abstract}
\classification{35J05, 35P15, 47B44}
%
\section{Introduction} \label{sekcejedna}
%
One of the first observations of purely real spectrum in a non-self-adjoint Schr\"{o}dinger operator occured in \cite{caliceti1980} by Caliceti et al. The authors studied the class of operators $-\mathrm{d}^2/\mathrm{d} x^2 + x^2 + \beta x^{2n+1}$ on $L^2(\RR)$ for a general complex $\beta$ and noticed that the spectrum is real provided $\arg(\beta) = \pi/2$ and $\beta$ is sufficiently small. This property was later attributed to the $\PT$-symmetry of the considered operator. The so-called $\PT$-symmetric quantum mechanics originated with the numerical observation of a purely real spectrum of an imaginary cubic oscillator Hamiltonian \cite{Bender-1998} and rapidly developed thenceforth. See e.g. \cite{Bender-2007, Mostafazadeh-2010} and references therein for a survey of papers in this area. The $\PT$-symmetry property of an operator $H$ should be understood in this paper as the invariance of $H$ with respect to the space inversion and the time reversal on the Hilbert space $L^2(\RR)$, i.e.
\beq \label{PTsymmetry}
[H,\PT] = 0
\eeq
in the operator sense, where $(\mathcal{P} \psi)(x) := \psi(-x)$ stands for spatial reflection and $(\T \psi)(x) := \overline{\psi(x)}$ stands for time reversal in quantum mechanics. Such operator possesses a relevant physical interpretation as an observable in quantum mechanics provided it is similar to a self-adjoint operator
\beq \label{similarity}
h=\Omega H \Omega^{-1},
\eeq
where $\Omega$ is a bounded and boundedly invertible operator. Then it is ensured that the spectra of $h$ and $H$ are identical and that the corresponding families of eigenfunctions share essential basis properties \cite{KrejcirikViola}. The similarity to a self-adjoint operator is in fact equivalent to the quasi-self-adjointness of $H$,
\beq \label{quasiH}
H^* \Theta = \Theta H,
\eeq
where the operator $\Theta$ is positive, bounded and boundedly invertible \cite{hussein, Scholtz-1992}. It is often called a metric, since the operator $H$ can be seen as self-adjoint in the space with the modified scalar product $(\cdot, \Theta \cdot)$. The equivalence can be easily seen from the decomposition of a positive operator $\Theta = \Omega^* \Omega$ \cite[Prop. 1.8]{kapitola}.
\paragraph{}In recent years it has been shown that the spectrum is not necessarily the best notion to describe properties of a non-self-adjoint operator and the use of $\eps$-pseudospectrum, denoted here $\sigma_{\eps}(H)$ and defined as
\beq
\sigma_{\varepsilon}(H) := \left\{ \lambda \in \CC \,\left|\, \left\|(H-\lambda)^{-1}\right\|> \varepsilon^{-1}\right.\right\},
\eeq
was suggested instead \cite{davies99, Dencker, henry, kapitola, KrejcirikViola, Krejcirik12, trefethen}. In \cite{Krejcirik12} authors studied the operator $-\mathrm{d}^2/\mathrm{d} x^2 + \ii x^3$ and derived completeness of its eigenfunctions, found the bounded metric operator $\Theta$ and proved that it can not have a bounded inverse. These results were further supplemented in \cite{KrejcirikViola} where the existence of points in the pseudospectrum far from the spectrum was established. This paper aims to apply the methods used in these papers to the operator $-\mathrm{d}^2/\mathrm{d} x^2 + x^2 +\ii x^3$, whose several properties were investigated e.g. in \cite{caliceti1980, delabaere, fernandez, mezincescu}. Our aim is to establish results which can be directly extended to the more general case $-\dd^2/\mathrm{d} x^2 + x^2 + \ii x^{2n+1}$, $n \geq 1$. We choose to study the case $n=1$ to show its relation to the famous imaginary cubic oscillator.
\paragraph{}Let us consider the Hilbert space $L^2(\RR)$ and define the operator $H$ acting on its maximal domain:
\beq \label{Ham}
\bal
H &:= -\frac{\mathrm{d}^2}{\mathrm{d} x^2} + x^2 + \ii x^3,\\
\Dom(H) &:= \left\{ \psi \in W^{2,2}(\RR)\,\left|\, -\frac{\mathrm{d}^2 \psi}{\mathrm{d} x^2} + x^2 \psi + \ii x^3\psi \in L^2(\RR)\right.\right\}.
\eal
\eeq
It was shown in \cite{caliceti1980} that $\Dom(H)$ coincides with $\left\{ \psi \in W^{2,2}(\RR)\,\left|\, x^3\psi \in L^2(\RR)\right.\right\}$ and that $H$ is closed. Furthermore, it is an operator with compact resolvent and therefore its spectrum is discrete (i.e.\ consists of isolated eigenvalues of finite algebraic multiplicity). The reality and the simplicity of the eigenvalues was established in \cite{shin02}. Using the approach of \cite[Sec.~ VII.2]{EE} shows that $H$ coincides with the closure of \eqref{Ham} defined on smooth functions with compact support and that it is an m-accretive operator. Recall that this means that $H$ is closed and that $\{\lambda \in \CC\mid \Re \lambda < 0 \} \subset \rho (H)$ and $\|(H - \lambda)^{-1}\| \leq 1/|\Re \lambda|$ for $\Re \lambda < 0$. In this paper we contribute to these results with showing the non-triviality of the pseudospectrum of $H$ and demonstrating its several consequences. The main results are summarised in the following theorem.
\begin{thm} \label{vetavet}
Let $H$ be the operator defined in \eqref{Ham}. Then:
\begin{enumerate}
\item The eigenfunctions of $H$ form a complete set in $L^2(\RR)$.
\item The eigenfunctions of $H$ do not form a (Schauder) basis in $L^2(\RR)$.
\item For any $\delta >0$ there exist constants $A, B >0$ such that for all $\eps>0$ small,
\beq
\left\{ \lambda\in\CC\,\left|\, |\lambda| > A,\, |\arg \lambda| < \arctan \Re \lambda - \delta,\, |\lambda| \geq B \left( \log\frac{1}{\eps}\right)^{6/5}\right.\right\} \subset \sigma_{\eps}(H).
\eeq
\item $H$ is not similar to a self-adjoint operator via bounded and boundedly invertible transformation
\item $H$ is not quasi-self-adjoint with a bounded and boundedly invertible metric.
\item $-\ii H$ is not a generator of a bounded semigroup.
\end{enumerate}
\end{thm}
We can see that for any $\eps$ the pseudospectrum contains complex points with positive real part, non-negative imaginary part and large magnitude. This result is in particular important in view of the characterisation of pseudospectrum \eqref{smallpert} -- it implies the existence of pseudomodes very far from the spectrum. This non-trivial behaviour of the pseudospectrum was without details announced in \cite{Krejcirik12}. A numerical computation of several of the pseudospectral lines of $H$ can be seen in Figure \ref{numerika}. As a consequence of the last point in Theorem \ref{vetavet}, the time-dependent Schr\"{o}dinger equation with $H$ does not admit a bounded time-evolution. For more details about establishing a time-evolution of an unbounded non-self-adjoint operator we refer to recent papers \cite{joe1, joe2} and to references therein.
\begin{figure}[ht]
\centering
\includegraphics[height=3in]{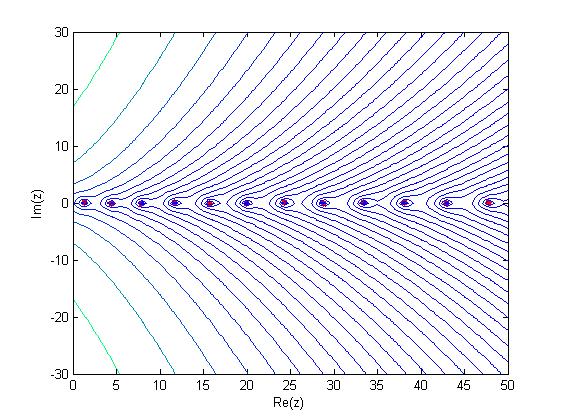}
\caption[Pseudospectrum]{Spectrum (red dots) and $\eps$-pseudospectra (enclosed by blue-green contour lines) of harmonic oscillator with imaginary cubic potential. The border of the $\eps$-pseudospectrum is plotted for the values $\eps = 10^{-7}, 10^{-6.75}, 10^{-6.5}, \dots, 10^1$, the green contour lines correspond to large values of $\eps$, the blue ones correspond to smaller values of $\eps$. We notice that for each $\eps$ from the selected range the contour lines quickly diverge and therefore the corresponding $\eps$-pseudospectrum contains points very far from the real axes. More details about the used computational method can be found in \cite{trefethenMATLAB}.}
\label{numerika}
\end{figure}
\paragraph{}This paper is organised as follows. In Section \ref{sekcedva} we formulate some properties of pseudospectrum to emphasize its importance in the study of non-self-adjoint operators. In Section \ref{sekcetech} we develop a semiclassical technique applicable in the study of pseudospectrum of the present model. The proof of the main theorem of the paper about pseudospectrum and eigenfunctions of $H$ can be found in Section \ref{sekcetri}. The Section \ref{sekcectyri} is devoted to a discussion of the results and of their consequences.
%
\section{General aspects of the pseudospectrum} \label{sekcedva}
%
The definition of the pseudospectrum and some of its most prominent properties are presented in this section. The focus is on properties related to this paper, which were already highlighted in \cite{KrejcirikViola}, where the authors dealt with similar problems. The presented list is far from complete and we refer to the monographs \cite{Davies07, trefethen} for more details on this subject.
\paragraph{}Let $H$ be a closed densely defined operator on a complex Hilbert space $\hilb$. Its spectrum $\sigma(H)$ is defined as the set of complex points $\lambda$ for which the operator $(H-\lambda)^{-1}$ does not exist or is not bounded on $\hilb$. The complement of this set in $\CC$ is called the resolvent set of $H$. It is a well known fact that the spectrum of a bounded linear operator is contained in the closure of its numerical range $\Theta(H) = \left\{(\psi,H\psi)\mid \psi \in \hilb, \|\psi\| = 1 \right\}$. Moreover, this holds for closed unbounded operators as well, provided the exterior of the numerical range in $\CC$ is a connected set and has a non-empty intersection with the resolvent set of $H$. The $\varepsilon$-pseudospectrum (or simply pseudospectrum) of $H$ is defined as
\beq
\sigma_{\varepsilon}(H) := \left\{ \lambda \in \CC \,\left|\, \left\|(H-\lambda)^{-1}\right\|> \varepsilon^{-1}\right.\right\},
\eeq
with the convention that $\left\|(H-\lambda)^{-1}\right\| = +\infty$ for $\lambda \in \sigma(H)$. In other words, $\sigma(H) \subset \sigma_{\varepsilon}(H)$ for every $\varepsilon$ from the definition and from the inequality $\left\|(H-\lambda)^{-1}\right\| \geq \mathrm{dist}\left(\lambda,\sigma(H)\right)^{-1}$ we can easily see that also an $\varepsilon$-neighbourhood of the spectrum is contained in the pseudospectrum. Similarly as in the previous case, if the exterior of the numerical range in $\CC$ is a connected set and has a non-empty intersection with the resolvent set of $H$, the pseudospectrum is in turn contained in the $\varepsilon$-neighbourhood of the numerical range, i.e.\ altogether we have
\beq
\left\{\lambda \in \CC \,\left|\, \mathrm{dist}(\lambda,\sigma(H))< \eps \right.\right\} \subset \sigma_{\varepsilon}(H) \subset \left\{\lambda \in \CC \,\left|\, \mathrm{dist}(\lambda,\overline{\Theta(H)})< \eps \right.\right\}.
\eeq
\paragraph{}Perhaps the most striking property of pseudospectrum is provided by result sometimes known as Roch-Silberman theorem \cite{roch}. The $\eps$-pseudospectrum of $H$ may be expressed via the spectra of all perturbations of $H$ of size less than $\eps$:
\beq \label{smallpert}
\sigma_{\eps}(H) = \bigcup_{\|V\|<\eps} \sigma(H + V).
\eeq
This result is especially important in the study of non-self-adjoint operators. For operators with highly non-trivial pseudospectrum (i.e.\ not contained in some bounded neighbourhood of the spectrum) it reveals their spectral instability with respect to small perturbations. It also shows a difficulty in the numerical study of operators with wild pseudospectra -- small rounding errors can lead to computing (false) eigenvalues, which are in fact very far from the true spectrum.
\paragraph{}The pseudospectrum can also by characterised as the set of all points of the spectrum and of all pseudoeigenvalues (or approximate eigenvalues), i.e.
\beq \label{pseudomode}
\sigma_{\varepsilon}(H)= \left\{\lambda \in \CC \mid \lambda \in \sigma(H) \vee \left(\exists \psi \in \Dom(H)\right)\left( \left\|(H-\lambda)\psi\right\|< \eps \|\psi\| \right)\right\}.
\eeq
Any $\psi$ satisfying the inequality in \eqref{pseudomode} is called a pseudoeigenvector (or pseudomode). It can be easily seen that pseudoeigenvalues can be turned into eigenvalues by a small perturbation. If $H$ were to represent a physical observable and $V$ its perturbation, this fact would cause some highly unintuitive behaviour of its energies.
%
\section{Semiclassical techniques} \label{sekcetech}
%
The use of semiclassical techniques in the study of non-self-adjoint operators was first suggested in \cite{davies99}, and the idea was further developed e.g. in \cite{Dencker, zworski-remark}. Let $H_h$ be an operator acting in $L^2(\RR)$ of the form
\beq
H_h := -h^2 \frac{\dd^2}{\dd x^2} + V_h(x).
\eeq
Here $V_h$ are analytic potentials in $x$ for all $h>0$ small enough which take the form $V_h(x) = V_0(x) + \widetilde V(x,h)$, where $\widetilde V(x,h) \rightarrow 0$ locally uniformly as $h \rightarrow 0$. This operator should be understood as some closed extension of an operator originally defined on $C_c^{\infty}(\RR)$. The following theorem is an analogue of \cite[Thm.~1]{davies99} for a potential depending on $h$. We postpone its proof into the following section.
\begin{thm} \label{veta}
Let $H_h$ be defined as above and let $\lambda$ be from the set
\begin{equation} \label{semiclass.pseudospec}
  \Lambda :=
  \left\{
  \xi^2 + V_h(x) \,\left|\, (x,\xi) \in \RR^{2}, \xi \,\Im V_h'(x) < 0
  \right.\right\},
\end{equation}
where the dash denotes standard differentiation with respect to $x$ in $\RR$. Then there exists some $C = C(\lambda) > 1$, some $h_0 = h_0(\lambda)>0$, and an $h$-dependent family of $C_c^{\infty}(\RR)$ functions $\{\psi_h\}_{0<h\leq h_0}$ with the property that, for all $0 < h \leq h_0$,
\beq \label{odhad2}
\left\|(H_h-\lambda)\psi_h\right\| < C^{-1/h}\|\psi_h\|.
\eeq
\end{thm}
The function $f(x,\xi):= \xi^2 + V_h(x)$ is called the symbol associated with $H_h$. Note that relation \eqref{pseudomode} gives us that $\lambda \in \sigma_{\varepsilon}(H_h)$ for all $\eps \geq C(\lambda)^{-1/h}$. Here $\eps$ can get arbitrarily close to $0$, provided $h$ is sufficiently small. The closure of $\Lambda$ is usually called the semiclassical pseudospectrum \cite{Dencker}. Application of Theorem \ref{veta} to non-semiclassical operators is sometimes possible by using scaling techniques and sending the spectral parameter to infinity. This is based on a more general principle that the semiclassical limit is equivalent to the high-energy limit after a change of variables.
\begin{proof}
The proof is inspired by the proof of \cite[Thm.~1]{KrejcirikViola}. We are interested in the case when $h$ is very close to $0$, during the course of the proof we are not going to stress every occasion when this plays a role. We can assume $h$ to be ``sufficiently small'' when necessary. Let $\lambda = \xi_0^2 + V_h(x_0)$ and assume $\xi_0 \neq 0, \Im V_h'(x_0) \neq 0$. Let us notice that $\lambda$ is dependent on $h$ from definition so changing $h$ in the course of our proof would mean changing $\lambda$ as well. This problem can be overcome by fixing $\lambda \in \Lambda$ and introduce a dependence of $x$ and $\xi$ on $h$ in such a way that $\lambda = \tilde\xi(h)^2 + V_h(\tilde x(h))$, where $\tilde\xi(h) \rightarrow \xi_0$ and $\tilde x(h) \rightarrow x_0$ as $h \rightarrow 0$. The existence of these functions is ensured by the implicit function theorem. Since we only need to find one function for which \eqref{odhad2} holds, the main idea is that the sought pseudomode will arise from JWKB approximation of the solution to $(H_h-\lambda)u = 0$ which takes the form
\beq
u(x,h):= \ee^{\ii \phi(x,h)/h} \sum_{j=0}^{N(h)} h^j a_j(x,h),
\eeq
where $a_j(x,h)$ are functions analytic near $x_0$. We follow here the procedure of constructing appropriate functions $\phi$ and $a_j$ as shown e.g.\ in \cite[Chap. 2]{dimassi}. The function $\phi$ should satisfy the eikonal equation
\beq
f(x,\phi'(x,h)) - \lambda =0,
\eeq
where $f$ is the symbol associated with $H_h$. (The dash denotes differentiation with respect to $x$.) From this equation immediately follows that $\phi'(x,h) = \pm \sqrt{\lambda - V_h(x)}$. The sign is determined by the condition $\xi \,\Im V_h'(x) < 0$ applied in the point $(x_0,\phi'(x_0,h))$ and remains the same for all $h$. Therefore the sign of $\phi'(x_0,h)$ should be opposite of the sign of $\Im V_h'(x_0)$. Therefore we get
\beq \label{phi}
\phi(x,h) = - \mathrm{sgn}\left(\Im V_h'(x_0)\right) \int_{0}^{x} \sqrt{\lambda - V_h(y)} \,\dd y.
\eeq
We need to check whether $\phi'$ is analytic near $x_0$ for $h$ small. From the assumption we know that $\Im V_h'(x_0) \neq 0$, so there exists $\delta > 0$ such that $\Im V_h'(\tilde x) \neq 0$ for $\tilde x \in [x_0-\delta,x_0+\delta]$. Then for every $\tilde x = x_0 + \varepsilon(h)$, where $0 < |\varepsilon| < \delta$ and $\eps(h) \rightarrow 0$ as $h \rightarrow 0$, we get
\beq
\Im V_h(\tilde x) - \Im \lambda = \Im V_h(\tilde x) - \Im V_0(x_0) = \varepsilon(h) \left( \Im V_0'(x_0) + \mathcal{O}(\varepsilon(h))\right) + \widetilde V(\tilde x,h)
\eeq
for $\eps$ going to $0$. Without loss of generality it is possible to assume $\Im V_h'(x_0) > 0$, therefore $\delta$ can be fixed so that $\Im V_0'(x_0) + \mathcal{O}(\varepsilon) > C'$ for some $C'>0$. Taking $h$ small, $\widetilde V(\tilde x,h)$ gets close to $0$ uniformly and $|\tilde x - x_0| < \delta$, thus $\Im V_h(\tilde x) - \Im \lambda > 0$ and consequently the square root in the definition of $\phi'$ is well-defined. The case $\Im V_h'(x_0) < 0$ is proven in the same manner. After a translation we can assume further on $x_0 = 0$.
\paragraph{}The equality
\beq
\ee^{-\ii \phi/h}\left(H_h - \lambda \right) \ee^{\ii \phi /h} = \frac{2 h}{\ii} \left( \phi' \frac{\dd}{\dd x} + \frac{1}{2}\phi''\right) - h^2 \frac{\dd^2}{\dd x^2}
\eeq
can be verified with a direct computation. If we set $a_j$ so that they satisfy the transport equations
\beq \label{prvni}
\bal
\phi'(x,h) a_0'(x, h) + \frac{1}{2} \phi''(x,h) a_0(x,h) &= 0, \\
\phi'(x,h) a_j'(x, h) + \frac{1}{2} \phi''(x,h) a_j(x,h) &= \frac{\ii}{2} a_{j-1}''(x,h)
\eal
\eeq
for $j > 0$, we get that 
\beq \label{pouzit}
\ee^{-\ii \phi(x,h)/h}\left(H_h - \lambda \right) \ee^{\ii \phi(x,h) /h}\left(\sum_{j=0}^{N} h^j a_j(x,h)\right) = -h^{N+2} a_N''(x,h).
\eeq
We can also set $a_0(x_0,h) = 1$ and $a_j(x_0,h) = 0$ for $j > 0$ and all $h$. The equations \eqref{prvni} can be then solved using the method of integrating factor as
\beq
\bal \label{aj}
a_0(x,h) &= \frac{\sqrt{\phi'(x_0,h)}}{\sqrt{\phi'(x,h)}},\\
a_j(x,h) &= \frac{1}{\sqrt{\phi'(x_0,h)}}\int_0^x \frac{\ii\, a_j''(y,h)}{2 \sqrt{\phi'(y,h)}} \,\dd y.
\eal
\eeq
These functions are well defined and analytic near $x_0$ thanks to analyticity of $\phi'$. We now proceed with estimates of the functions $a_j$. Note that since the potentials $V_h(x)$ are analytic, we can naturally extend them into the complex plane in the neighbourhood of $x_0 = 0$ and thus the same can be applied on $\phi$ and all $a_j$. Our goal is to arrive to the estimate 
\beq \label{estimate}
|a_j(x,h)| \leq C_1^{j+1} j^j
\eeq
for $C_1 > 0$ and  $x$ in some neighbourhood of the origin. Then we will be able to define the $h$-dependent function
\beq \label{fa}
a(x,h) := \sum_{0\leq j \leq (\ee C_1 h)^{-1}} h^j a_j(x,h),
\eeq
which is uniformly bounded analytic function on the set where \eqref{estimate} holds due to the absolute summability of the sum
\beq \label{suma}
|a(x,h)|\leq C_1\sum_{0\leq j \leq (\ee C_1 h)^{-1}} C_1^j h^j j^j \leq C_1 \sum_{0\leq j \leq (\ee C_1 h)^{-1}} \ee^{-j} < +\infty.
\eeq
In the following we will derive the estimate \eqref{estimate} for $a_j$ extended to the complex plane (further denoted as $a_j(z,h)$). With the natural choice of the norm
\beq
\|f\|_{B(R)} := \sup\left\{z \in B(R) \mid |z| < R\right\},
\eeq
where $B(R)$ is an open ball in the complex plane with center at $0$ and diameter $R$, we easily see that the estimate obtained for $\|a\|_{B(R)}$ will remained valid for $|a(x,h)|$ in some neighbourhood of the origin. We fix $R_0$ such that, on $B(R_0)$,  $\phi$ is analytic, $|\phi'|$ is bounded from below and above and $\Im \phi''(x,h) > 1/C_2$ for some $C_2 > 0$. We also employ Cauchy's estimate for the second derivative of an analytic bounded function $f$ defined on $B(R)$:
\beq \label{cauchy}
|f''(z)|\leq \frac{2 \|f\|_{B(R)}}{\left(R - |z|\right)^2}.
\eeq
With the use of the formula \eqref{aj} we obtain
\beq\bal
|a_{j}(z,h)| &= \left|\frac{1}{\sqrt{\phi'(z,h)}}\int_0^z \frac{\ii a_{j-1}''(\zeta,h)}{2 \sqrt{\phi'(\zeta,h)}} \,\dd \zeta\right|\\
&\leq \|(\phi'(\cdot,h))^{-1}\|_{B(R)} \int_0^{|z|} \frac{\|a_{j-1}(\cdot,h)\|_{B(R)}}{(R-t)^2} \, \dd t\\
&=\|(\phi'(\cdot,h))^{-1}\|_{B(R)} \|a_{j-1}(\cdot,h)\|_{B(R)} \left( \frac{1}{R - |z|} - \frac{1}{R}\right)\\
&= \frac{|z|}{R(R - |z|)}\|(\phi'(\cdot,h))^{-1}\|_{B(R)} \|a_{j-1}(\cdot,h)\|_{B(R)}
\eal\eeq
for $j = 0, 1, \dots$. We iterate these estimates on balls of radius $R_k := \left(1 - k/2j\right)R_0$, $k = 0, \dots, j-1$. Then we have for $|z| < R_{j}$
\beq
\frac{|z|}{R_k(R_k - |z|)} \leq \frac{|z|}{R_k(R_k - R_{k+1})} \leq \frac{4 j |z|}{R_0^2}.
\eeq 
Then it follows for $a_{k+1}$ that 
\beq
|a_{k+1}(z,h)| \leq \frac{4 j |z|}{R_0^2} \|(\phi'(\cdot,h))^{-1}\|_{B(R_0)} \|a_j(\cdot,h)\|_{B(R)}.
\eeq
Subsequently using these estimates for $k = 0, \dots, j-1$ and taking a supremum we obtain
\beq
\|a_j(\cdot,h)\|_{B(R_0/2)} \leq \|a_j(\cdot,h)\|_{B(R_j)} \leq \|a_0(\cdot,h)\|_{B(R_0)} \left(\frac{2 j}{R_0} \|(\phi'(\cdot,h))^{-1}\|_{B(R_0)}\right)^j.
\eeq
We see from \eqref{aj} and our choice of $R_0$ that $\|a_0(\cdot,h)\|_B(R_0) < C_3$ and from the uniform estimate of $|\phi(x,h)|$ from below that $\|(\phi'(\cdot,h))^{-1}\|_{B(R_0)} < C_4$, where the positive constants $C_3$ and $C_4$ does not depend on $h$. The desired estimate \eqref{estimate} then follows with the constant
\beq
C_1 := \max\left\{ C_3, \frac{2}{R_0} C_4\right\}.
\eeq

\paragraph{}We are now able to define the desired pseudomode as 
\beq \label{psi}
\psi_h(x) := \ee^{\ii \phi(x,h)/h}\chi(x) a(x,h),
\eeq
where $a(x,h)$ is the function defined in \eqref{fa} and $\chi \in C_c^{\infty}(\RR)$ such that it is identically equal to $1$ in some neighbourhood of $0$ and its support lies inside the interval $(-R_0/2, R_0/2)$. We divide the calculation of the norm in \eqref{odhad2} as follows:
\beq \label{summand}
\left\|(H_h-\lambda)\psi_h\right\| = \left\|\chi(H_h-\lambda)\ee^{\ii \phi/h} a \right\| + \left\|\left[H_h-\lambda,\chi\right]\ee^{\ii \phi/h} a\right\|.
\eeq
First we focus on the first summand. Since $\phi(0,h) = 0$, $\phi'(0, h)$ is real and $\Im \phi''(x, h) > 1/C_2$ holds, we have 
\beq \label{minusexp}
\left|\ee^{\ii \phi(x,h)/h}\right|\leq \exp\left(-\frac{x^2}{2 C_2 h}\right)
\eeq
for all $x \in \mathrm{supp} \chi$. Since $\left|\ee^{\ii \phi/h}\right| > 1$ on $\mathrm{supp} \chi$, we can use \eqref{pouzit} to estimate
\beq
\left\|\chi(H_h-\lambda)\ee^{\ii \phi/h} a \right\| \leq \left\|\chi \ee^{\ii \phi/h}(H_h-\lambda)\ee^{\ii \phi/h} a \right\| = \|h^{N+2}a_N''\chi\|,
\eeq
where $N = N(h) = \lfloor (\ee C_1 h)^{-1}\rfloor$. (Here $\lfloor x \rfloor$ denotes the floor function.) Using the estimate from \eqref{estimate} and the Cauchy's estimate \eqref{cauchy} we obtain for all $x \in \mathrm{supp} \chi$ that $|h^{N+2}a_N''(x,h)|\leq C \ee^{-1/(Ch)}$ for $C > 0$ independent of $h$. From this the estimate
\beq \label{pozdeji}
\left\|\chi(H_h-\lambda)\ee^{\ii \phi/h} a \right\| \leq C \ee^{-1/(Ch)}
\eeq
follows. To estimate the second summand in \eqref{summand} we directly calculate
\beq \label{treti}
\left[H_h-\lambda,\chi\right]\ee^{\ii \phi/h} a = -h^2 \ee^{\ii \phi/h} \left(\chi''a + 2\chi'\left(a' + \frac{\ii}{h}\phi'a \right)\right).
\eeq
Using \eqref{suma} we have uniform bounds on $a$ and thus on $a'$ after the use of the Cauchy's estimate \eqref{cauchy}, $\phi'$ is bounded by the choice of $R_0$, $\chi'$ and $\chi''$ are identically equal to $0$ on $\mathrm{supp} \chi$ and $\ee^{\ii \phi/h}$ is again bounded by \eqref{minusexp} we see that \eqref{treti} is in fact equal to $0$ on the neighborhood of $0$, where $\chi$ is constant.
\paragraph{}To complete the proof, it remains to show that $\psi_h$ defined in \eqref{psi} is not exponentially small. Since we have established the estimate \eqref{estimate} for $|x| < R_0/2$ and $0 \leq j \leq N = (\ee C_1 h)^{-1}$, we have the estimate
\beq
\left\|\sum_{j=1}^N h^j a_j(x,h)\right\|_{B(r)} \leq C r
\eeq
for $0 < r \leq r_0$, where $r_0$ is sufficiently small. We can take $r$ very small and fixed, so because $a_0(x,h)$ is close to $1$ and $\Im \phi(x,h)$ is close to $\Im \phi''(0,h)x^2/2$ for $x$ small, we obtain
\beq
\|u(\cdot,h)\|\geq\|u(\cdot,h)\|_{L^2((-r,r))}\geq \frac{1}{C}\left( \int_{-r}^r \exp\left(\frac{x^2}{C h}\right)\,\dd x\right)^{1/2}\geq \frac{1}{C} h^{1/4}.
\eeq
\end{proof}
%
\section{The proof of Theorem \ref{vetavet}} \label{sekcetri}
%
For the sake of clarity we choose to divide the proof into several lemmas.
\begin{lem}
The eigenfunctions of $H$ form a complete set in $L^2(\RR)$.
\end{lem}
\begin{proof}
Let us first briefly recall that completeness of $\{\psi_k\}_{k=1}^{+\infty}$ means that the span of $\psi_k$ is dense in $L^2(\RR)$. Since $H$ is m-accretive, its resolvent is m-accretive as well. It is also a Hilbert-Schmidt operator \cite{caliceti1980} and the application of \cite[Thm.~1.3]{Almog14} yields that it is trace class as well. The completeness of its eigenfunctions follows from \cite[Thm.~X.3.1]{Gohberg90}. The completeness of eigenfunctions of $H$ then follows from the application of the spectral mapping theorem \cite[Thm.~IX.2.3]{EE}.
\end{proof}
\begin{lem}
For any $\delta >0$ there exist constants $C_1,C_2 >0$ such that for all $\eps>0$ small,
\beq
\left\{ \lambda\in\CC\,\left|\, |\lambda| > C_1,\, |\arg \lambda| < \frac{\pi}{2} - \delta,\, |\lambda| \geq C_2 \left( \log\frac{1}{\eps}\right)^{6/5}\right.\right\} \subset \sigma_{\eps}(H).
\eeq
\end{lem}
\begin{proof}
Using the unitary transformation
\beq
\left(U \psi \right)(x) := \tau^{1/2}\, \psi(\tau x ) 
\eeq
the semiclassical analogue of $H$ is introduced:
\beq
U H U^{-1} = \tau^3 H_{h},
\eeq
where 
\beq
H_{h} := - h^2 \frac{\mathrm{d}^2}{\mathrm{d}x^2} + h^{2/5}x^2 + \ii x^3
\eeq
and $h := \tau^{-5/2}$. For the set $\Lambda$ from Theorem \ref{veta} holds $\Lambda = \left\{ \lambda \in \CC \mid \Re \lambda > 0, |\arg \lambda| < \arctan \Re \lambda \right\}$. This theorem gives us that for any $\lambda \in \Lambda$, $n > 0$ and $h$ sufficiently small
\beq \label{odhad}
\left\|\left(H - \tau^3 \lambda\right)^{-1} \right\| = \tau^{-3} \left\|\left(H_h - \lambda\right)^{-1} \right\| > h^{6/5} C(\lambda)^{1/h}
\eeq
holds. Let us define the set
\beq
A_{\delta} = \left\{ \lambda\in \CC \,\left|\, |\lambda|=1, |\arg \lambda| < \arctan \Re \lambda - \delta\right.\right\}
\eeq
for any $\delta > 0$. Then we see from the inequality \eqref{odhad} that $\tau^3 A_{\delta} \subset \sigma_{\eps}(H)$ for every $\delta$ and every $\tau$ sufficiently large, in particular such that the inequality $\tau^{-3} C^{\tau^{5/2}} > \varepsilon^{-1}$ holds. We may then identify the points of $\Lambda$ in absolute value with $\tau^3$, i.e.\ $|\lambda| = \tau^3 = h^{-6/5}$. After we take logarithm of the lastly mentioned inequality and neglect the term $\log \tau^{-3}$ which is small compared to $\tau^{5/2}$ for $\tau$ large, the statement of the theorem follows after expressing the inequality in terms of $|\lambda|$.
\end{proof}
\begin{lem}
The eigenfunctions of $H$ do not form a (Schauder) basis in $L^2(\RR)$.
\end{lem}
\begin{proof}
Let us first recall that a Schauder basis is a set $\{ \psi_k\}_{k=1}^{+\infty}\subset \hilb$ such that for every element $\psi \in \hilb$ can be uniquely expressed as $\psi = \sum_{k=1}^{+\infty} \alpha_k \psi_k$, where $\alpha_k \in \CC$ for $k=1, 2,\dots$. From the inequality \eqref{odhad} we can clearly see that the norm of the resolvent $(H-\lambda)^{-1}$ shoots up exponentially fast for $|z|$ large. Therefore the eigenfunctions of $H$ cannot be tame by \cite[Thm.~3]{Davies-00}. Specifically, if we arrange the eigenvalues $\lambda_k$ of $H$ in increasing order, the norm of spectral projection $P_k$ corresponding to $\lambda_k$ cannot satisfy
\beq
\|P_k\| \leq a k^{\alpha}
\eeq
for some $a, \alpha$ and all $k$. Therefore $\{\psi_k\}_{k=1}^{+\infty}$ cannot form a basis.
\end{proof}
\begin{lem}
$- \ii H$ is not a generator of a bounded semigroup.
\end{lem}
\begin{proof}
As in the previous proof, since the norm of resolvents grows exponentially for $|z|$ large, the claim follows from \cite[Thm.~8.2.1]{Davies07}. 
\end{proof}
The following result is a direct consequence of several propositions about operators with non-trivial pseudospectra from \cite{KrejcirikViola} which apply to $H$ as well. We summarise them and provide a compact proof.
\begin{lem}
$H$ is not similar to a self-adjoint operator via bounded and boundedly invertible transformation and $H$ is not quasi-self-adjoint with a bounded and boundedly invertible metric.
\end{lem}
\begin{proof}
If $H$ were similar to a self-adjoint operator $h$ as in \eqref{similarity}, its pseudospectrum would have to satisfy
\beq
\sigma_{\varepsilon/\kappa}(H) \subset \sigma_{\varepsilon}(h) \subset \sigma_{\varepsilon \kappa}(H),
\eeq
where $\kappa = \|\Omega\|\|\Omega^{-1}\|$. However, since the pseudospectrum of $h$ is just the $\eps$-neighbourhood of its spectrum, it cannot contain arbitrarily large points as $\varepsilon/\kappa$-pseudospectrum of $H$ does. The claim about the quasi-self-adjointness \eqref{quasiH} follows from the already established equivalence from the decomposition $\Theta = \Omega^* \Omega$. 
\end{proof}
%
\section{Summary} \label{sekcectyri}
%
The harmonic oscillator coupled with an imaginary cubic oscillator potential was the main subject of interest of the present paper and we aimed to provide a detailed study of its basis and pseudospectral properties. The pseudospectrum of $H$ exhibits wild properties and contains points very far from the spectrum, which can be turned into true eigenvalues by a small perturbation of the operator. As a consequence, the eigenfunctions of $H$ do not form a Schauder basis, although they form a dense set in $L^2(\RR)$. The semigroup associated with the time-dependent Sch\"odinger equation then does not have an expansion in the basis of eigenfunctions and does not admit a bounded time-evolution. The non-trivial pseudospectrum also implies that the considered operator does not have any bounded and boundedly invertible metric and thus it cannot be faithfully represented by any self-adjoint operator in the framework of standard quantum mechanics. In conclusion let us note that all results of this paper can be directly generalised to potentials of the type $x^2 + \ii x^{2n+1}$, since all previous cited results apply to this more general case as well.
%
\section*{Acknowledgements}
%
The research was supported by the Czech Science Foundation within the project 14-06818S and by Grant Agency of the Czech Technical University in Prague, grant No. SGS13/217/OHK4/3T/14. The author would like to express his gratitude to David Krej\v ci\v r\' ik, Petr Siegl, Joseph Viola and Milo\v s Tater for valuable discussions.
%
\bibliographystyle{acm}

%
\end{document}